\documentclass[11pt]{amsart}
\usepackage{amsmath,amsfonts,amscd,amssymb}
\usepackage[colorlinks=true, linkcolor=blue,citecolor=blue, urlcolor=blue]{hyperref}
\usepackage{multirow}
\usepackage{tableau}
\usepackage{subcaption}
\usepackage[alphabetic]{amsrefs}
\usepackage{tikz}
\usepackage{hhline}
\usepackage{enumitem}
\usepackage[margin=1in]{geometry}
\usepackage[ordering=Bourbaki,edge-length=0.6 cm,edge/.style=black,indefinite-edge={draw=black,fill=white,densely dashed},indefinite-edge-ratio=2,
mark=o,arrow style={black,
length=2mm,
width=3mm
}]{dynkin-diagrams}

\def\newthm#1#2{\newtheorem{#1}[dummy]{#2}%
  \expandafter\def\csname#2\endcsname##1{\hyperref[#1:##1]{#2~\ref*{#1:##1}}}}
\newthm{thm}{Theorem}
\newthm{lemma}{Lemma}
\newthm{prop}{Proposition}
\newthm{cor}{Corollary}
\newthm{conj}{Conjecture}
\newthm{cons}{Consequence}
\newthm{fact}{Fact}
\newthm{obs}{Observation}
\newthm{guess}{Guess}
\theoremstyle{definition}
\newthm{defn}{Definition}
\newthm{remark}{Remark}
\newthm{example}{Example}
\newthm{question}{Question}
\newthm{notation}{Notation}

\makeatletter
\def\namedlabel#1#2{\begingroup
    #2%
    \def\@currentlabel{#2}%
    \phantomsection\label{#1}\endgroup
}
\makeatother

\newcommand{\Section}[1]{\hyperref[sec:#1]{Section~\ref*{sec:#1}}}
\newcommand{\Table}[1]{\hyperref[tab:#1]{Table~\ref*{tab:#1}}}
\newcommand{\eqn}[1]{\hyperref[eqn:#1]{(\ref*{eqn:#1})}}
\newcommand{\Figure}[1]{\hyperref[fig:#1]{Figure~\ref*{fig:#1}}}

\DeclareMathOperator{\Gr}{Gr}

\DeclareMathOperator{\LG}{LG}

\DeclareMathOperator{\OG}{OG}

\DeclareMathOperator{\Pic}{Pic}

\DeclareMathOperator{\ch}{\ch}

\usepackage{bm}

\newcommand{\bP}{{\mathbb P}}
\renewcommand{\P}{{\mathbb P}}

\newcommand{\cP}{{\mathcal P}}
\newcommand{\cD}{{\mathcal D}}

\newcommand{\ignore}[1]{}

      \DeclarePairedDelimiter{\bracks}{\lbrack}{\rbrack}

\begin{document}

\title[The Isomorphism Problem for cominuscule Schubert Varieties]
{The Isomorphism Problem for cominuscule Schubert Varieties}

\date{\today}

\title[Cominuscule Schubert varieties]{The isomorphism problem for cominuscule Schubert varieties}

\author{Edward Richmond}
\address{Department of Mathematics, Oklahoma State University, Stillwater, OK, USA}
\email{edward.richmond@okstate.edu}
\author{Mihail Țarigradschi}
\address{Department of Mathematics, Rutgers University, Piscataway, NJ, USA}\email{mt994@math.rutgers.edu}
\author{Weihong Xu}
\address{Department of Mathematics, Virginia Tech, Blacksburg, VA, USA}
\email{weihong@vt.edu}


\begin{abstract}
        Cominuscule flag varieties generalize Grassmannians to other Lie types. Schubert varieties in cominuscule flag varieties are indexed by posets of roots labeled long/short. These labeled posets generalize Young diagrams. We prove that Schubert varieties in potentially different cominuscule flag varieties are isomorphic as varieties if and only if their corresponding labeled posets are isomorphic, generalizing the classification of Grassmannian Schubert varieties using Young diagrams by the last two authors. Our proof is type-independent.
\end{abstract}

\maketitle

\section{Introduction}

Cominuscule flag varieties correspond to algebraic varieties that admit the structure of a compact Hermitian symmetric space and have been studied extensively due their shared properties with Grassmannians \cites{BP99,Pe09,ThYo09,BCMP18,MR4099894,buch2022positivity}. These varieties come in five infinite families and two exceptional types and are determined by a pair \((\cD,\gamma)\) of a Dynkin diagram \(\cD\) of a reductive Lie group and a cominuscule simple root \(\gamma\). See \Table{comin} for a classification of cominuscule flag varieties. Let \(X\) denote the cominuscule flag variety corresponding to \((\cD,\gamma)\) and \(R\) denote the root system of the Dynkin diagram \(\cD\).  Set \(\cP_X\coloneqq \{\alpha\in R: \alpha\geq\gamma\}\) with the partial order $\alpha\leq \beta$ if $\beta-\alpha$ is a non-negative sum of simple roots, and give \(\cP_X\) a labeling of long/short roots. By~\cite{buch2016k}*{Theorem 2.4}, Schubert varieties in \(X\) are indexed by lower order ideals in \(\cP_X\), generalizing the fact that Schubert varieties in a Grassmannian are indexed by Young diagrams.

Our main result \Theorem{main} is a combinatorial criterion for distinguishing isomorphism classes of Schubert varieties coming from cominuscule flag varieties.

\begin{thm}\label{thm:main}
        Let $X_\lambda\subseteq X$ and $Y_{\mu}\subseteq Y$ be cominuscule Schubert varieties indexed by lower order ideals \(\lambda\subseteq\cP_X\) and \(\mu\subseteq\cP_Y\), respectively.  Then $X_\lambda$ and $Y_{\mu}$ are algebraically isomorphic if and only if $\lambda$ and $\mu$ are isomorphic as labeled posets.
\end{thm}

For illustrative examples of \Theorem{main}, see \Section{examples}.

Since Grassmannians are cominuscule flag varieties, \Theorem{main} extends the work of Țarigradschi and Xu in \cite{Tarigradschi-Xu22}, where they prove two Grassmannian Schubert varieties are isomorphic if and only if their Young diagrams are the same or the transpose of each other. Other related works include Richmond and Slofstra's characterization of the isomorphism classes of Schubert varieties coming from complete flag varieties in \cite{Richmond-Slofstra21} using Cartan equivalence. However, they also note that Cartan equivalence is neither necessary nor sufficient to distinguish Schubert varieties in partial flag varieties. A class of smooth Schubert varieties in type \(A\) partial flag varieties are classified by Develin, Martin, and Reiner in \cite{develin_martin_reiner_2007}. Yet many Schubert varieties are singular, with the first example being the Schubert divisor in the Grassmannian \(\Gr(2,4)\).

We discuss preliminaries in \Section{prelim}, and then in \Section{proof}, we prove \Theorem{main} and illustrate it with examples. Our proof is type-independent and employs several new techniques. To show that the labeled poset \(\lambda\) depends only on the isomorphism class of the Schubert variety \(X_\lambda\), we construct it from the effective cone in the Chow group of \(X_\lambda\) and intersection products of classes in this cone with the unique effective generator of the Picard group of the variety. To prove the converse, we embed each Schubert variety \(X_\lambda\) in a ``minimal" cominuscule flag variety uniquely determined by the labeled poset \(\lambda\).  Assuming that the labeled posets \(\lambda\subseteq\cP_X\) and \(\mu\subseteq\cP_Y\) are isomorphic, we construct an explicit isomorphism between the Dynkin diagrams of the ``minimal'' cominuscule flag varieties determined by \(\lambda\) and \(\mu\).  The corresponding flag variety isomorphism identifies the Schubert varieties \(X_\lambda\) and \(Y_\mu\).

\begin{table}[!ht]
  \caption{Cominuscule flag varieties. Cominuscule roots are denoted by the filled-in circles. This table is a modification of \cite[Table 1]{buch2016k}.}\label{table:comin}
  \centering
  \begin{tabular}{|l|l|}
    \hline
                                                                                                                                                  & \vspace{-3mm}                  \\
    \multirow{3}{*}{\ \begin{dynkinDiagram}[text style/.style={scale=0.9},labels={1,2,n-1,n},scale=2]
                          {A}{**.**}
                        \end{dynkinDiagram}}                                & Type \(A_n\) (\(n\geq 1\)):                                                                  \\ & \(A_n/P_m = \Gr(m,n+1)\) \\ & Grassmannian of type A.\\
                                                                                                                                                  & \vspace{-3mm}                  \\
    \hline
                                                                                                                                                  & \vspace{-3mm}                  \\
    \multirow{2}{*}{\ \begin{dynkinDiagram}[text style/.style={scale=0.9},labels={1,2,n-2,n-1,n},scale=2]
                          {B}{*o.ooo}
                        \end{dynkinDiagram}}                                & Type \(B_n\) (\(n\geq 2\)):                                                                  \\ & $B_n/P_1 = Q^{2n-1}$ \ Odd quadric. \\
                                                                                                                                                  & \vspace{-3mm}                  \\
    \hline
                                                                                                                                                  & \vspace{-3mm}                  \\
    \multirow{3}{*}{\ \begin{dynkinDiagram}[text style/.style={scale=0.9},labels={1,2,n-2,n-1,n},scale=2]
                          {C}{oo.oo*}
                        \end{dynkinDiagram}}                                & Type \(C_n\) (\(n\geq 3\)):                                                                  \\ & $C_n/P_n = \LG(n,2n)$ \\
                                                                                                                                                  & Lagrangian Grassmannian.       \\
                                                                                                                                                  & \vspace{-3mm}                  \\
    \hline
                                                                                                                                                  & \vspace{-3mm}                  \\
    \multirow{6}{*}{\ \begin{dynkinDiagram}[text style/.style={scale=0.9},labels={1,2,n-3,n-2,n-1,n},label directions={,,,right,,},scale=2]
                          {D}{*o.oo**}
                        \end{dynkinDiagram}} &                                   \\ & Type \(D_n\) (\(n\geq 4\)):\\ &  $D_n/P_1 = Q^{2n-2}$ \ Even quadric.\\ & $D_n/P_{n-1} \cong D_n/P_n = \OG(n,2n)$ \ \\
                                                                                                                                                  & Max.\ orthogonal Grassmannian. \\
                                                                                                                                                  &                                \\
                                                                                                                                                  & \vspace{-3mm}                  \\
    \hline
                                                                                                                                                  & \vspace{-3mm}                  \\
    \multirow{5}{*}{\ \begin{dynkinDiagram}[text style/.style={scale=0.9},labels={1,2,3,4,5,6},scale=2]
                          {E}{*oooo*}
                        \end{dynkinDiagram}}                              &                                                                \\ & Type \(E_6\): \\ & $E_6/P_1 \cong E_6/P_6$ \\ &  Cayley plane. \\ & \\
                                                                                                                                                  & \vspace{-3mm}                  \\
    \hline
                                                                                                                                                  & \vspace{-3mm}                  \\
    \multirow{5}{*}{\ \begin{dynkinDiagram}[text style/.style={scale=0.9},labels={1,2,3,4,5,6,7},scale=2]
                          {E}{oooooo*}
                        \end{dynkinDiagram}}                              &                                                                \\ & Type \(E_7\): \\ & $E_7/P_7$ \\ & Freudenthal variety.\\ & \\
    \hline
  \end{tabular}
  \label{tab:comin}
\end{table}

\def\vmm#1{\vspace{#1mm}}
\begin{table}
  \caption{The labeled poset \(\cP_X\) for a cominuscule flag variety \(X\). Each element in \(\cP_X\) is drawn as a box, and boxes decorated with an ``\(s\)'' correspond to short roots.  The partial order on boxes is given by $\alpha \leq \beta$ if and only if $\alpha$ is weakly north-west of $\beta$. This table is a modification of \cite[Table 1]{buch2022positivity}.}\label{table:poset}

  \begin{tabular}{|c|}
    \hline                                                                                                                                                    \\
    Grassmannian $\Gr(3,7)$ of type A                                                                                                                         \\
    \begin{dynkinDiagram}[text style/.style={scale=0.9},labels={1,2,3,4,5,6},scale=2]{A}{oo*ooo}\end{dynkinDiagram}                                \\
    $\tableau{12}{
    [a]   & [a]    & [a]    & [a]                                                                                                                             \\
    [a]   & [a]    & [a]    & [a]                                                                                                                             \\
    [a]   & [a]    & [a]    & [a]
    }$                                                                                                                                                        \\  \\
    \hline                                                                                                                                                    \\

    Odd quadric $Q^{11} \subset \bP^{12}$                                                                                                                     \\
    \multirow{2}{*}{\ \begin{dynkinDiagram}[text style/.style={scale=0.9},labels={1,2,3,4,5,6},scale=2]
                          {B}{*ooooo}
                        \end{dynkinDiagram}}                                             \\ \\
    $\tableau{12}{
    [a]{} & [a]{}  & [a]    & [a]    & [a]    & [a]{s} & [a] & [a]{} & [a]{} & [a]{} & [a]{}
    }$                                                                                                                                                        \\ \\
    \hline                                                                                                                                                    \\
    Lagrangian Grassmannian $\LG(6,12)$                                                                                                                       \\
    \multirow{2}{*}{\ \begin{dynkinDiagram}[text style/.style={scale=0.9},labels={1,2,3,4,5,6},scale=2]
                          {C}{ooooo*}
                        \end{dynkinDiagram}}                                             \\ \\
    $\tableau{12}{
    [a]{} & [a]{s} & [a]{s} & [a]{s} & [a]{s} & [a]{s}                                                                                                        \\
          & [a]{}  & [a]{s} & [a]{s} & [a]{s} & [a]{s}                                                                                                        \\
          &        & [a]{}  & [a]{s} & [a]{s} & [a]{s}                                                                                                        \\
          &        &        & [a]{}  & [a]{s} & [a]{s}                                                                                                        \\
          &        &        &        & [a]{}  & [a]{s}                                                                                                        \\
          &        &        &        &        & [a]{}
    }$                                                                                                                                                        \\ \\
    \hline                                                                                                                                                    \\
    Max.\ orthog.\ Grassmannian $\OG(6,12)$                                                                                                                   \\
    \begin{dynkinDiagram}[text style/.style={scale=0.9},labels={1,2,3,4,5,6},label directions={,,,right,,},scale=2]{D}{ooooo*}\end{dynkinDiagram} \\
    $\tableau{12}{
    [a]   & [a]    & [a]    & [a]    & [a]                                                                                                                    \\
          & [a]    & [a]    & [a]    & [a]                                                                                                                    \\
          &        & [a]    & [a]    & [a]                                                                                                                    \\
          &        &        & [a]    & [a]                                                                                                                    \\
          &        &        &        & [a]
    }$                                                                                                                                                        \\ \\
    \hline
  \end{tabular}
  \begin{tabular}{|c|}
    \hline                                                                                                                                          \\
    Even quadric $Q^{10} \subset \bP^{11}$                                                                                                          \\
    \multirow{4}{*}{\ \begin{dynkinDiagram}[text style/.style={scale=0.9},labels={1,2,3,4,5,6},label directions={,,,right,,},scale=2]
                          {D}{*ooooo}
                        \end{dynkinDiagram}} \\ \\ \\ \\
    $\tableau{12}{
    [a] & [a] & [a] & [a] & [a]                                                                                                                     \\
        &     &     & [a] & [a] & [a] & [a] & [a]
    }$                                                                                                                                              \\ \\
    \hline                                                                                                                                          \\

    Cayley Plane $E_6/P_6$                                                                                                                          \\
    \multirow{5}{*}{\ \begin{dynkinDiagram}[text style/.style={scale=0.9},labels={1,2,3,4,5,6},scale=2]
                          {E}{ooooo*}
                        \end{dynkinDiagram}}                                 \\ \\ \\ \\ \\
    $\tableau{12}{
    [a] & [a] & [a] & [a]                                                                                                                           \\
        &     & [a] & [a] & [a] & [a]                                                                                                               \\
        &     & [a] & [a] & [a] & [a]                                                                                                               \\
        &     &     &     & [a] & [a] & [a] & [a]
    }$                                                                                                                                              \\ \\
    \hline                                                                                                                                          \\
    Freudenthal variety $E_7/P_7$                                                                                                                   \\
    \multirow{5}{*}{\ \begin{dynkinDiagram}[text style/.style={scale=0.9},labels={1,2,3,4,5,6,7},scale=2]
                          {E}{oooooo*}
                        \end{dynkinDiagram}}                                 \\ \\ \\ \\ \\
    $\tableau{12}{
    [a] & [a] & [a] & [a] & [a] & [a]                                                                                                               \\
        &     &     & [a] & [a] & [a]                                                                                                               \\
        &     &     &     & [a] & [a] & [a]                                                                                                         \\
        &     &     &     & [a] & [a] & [a] & [a] & [a]                                                                                             \\
        &     &     &     & [a] & [a] & [a] & [a] & [a]                                                                                             \\
        &     &     &     &     &     &     & [a] & [a]                                                                                             \\
        &     &     &     &     &     &     &     & [a]                                                                                             \\
        &     &     &     &     &     &     &     & [a]                                                                                             \\
        &     &     &     &     &     &     &     & [a]
    }$                                                                                                                                              \\ \\
    \hline
  \end{tabular}\label{tab:poset}
\end{table}

\subsection*{Acknowledgements} We thank Anders Buch for helpful discussions and the anonymous referee for helpful comments. ER was supported by a grant from the Simons Foundation 941273. MȚ and WX were supported by NSF grant MS-2152316.

\section{Examples of \Theorem{main}}

\label{sec:examples}

For the following examples, recall that cominuscule Schubert varieties are indexed by lower order ideals in $\cP_X$. Examples of \(\cP_X\) are illustrated in \Table{poset}, where each element in \(\cP_X\) is drawn as a box, and boxes decorated with an ``\(s\)'' correspond to short roots.  The partial order on boxes is given by $\alpha \leq \beta$ if and only if $\alpha$ is weakly north-west of $\beta$. In particular, lower order ideals are given by subsets of boxes that are closed under moving to the north and west.

\begin{example}
        As illustrated below, transposing a Young diagram does not change the poset structure:
        \begin{equation*}
                \tableau{12}{
                [a]&[a]&[a]&[a]\\
                [a]&[a]\\
                [a]
                }\ \ \cong\ \
                \tableau{12}{
                [a]&[a]&[a]\\
                [a]&[a]\\
                [a]\\
                [a]
                }.
        \end{equation*}Therefore, two Grassmannian Schubert varieties are isomorphic if their indexing Young diagrams are the transpose of each other. Geometrically, this is related to the isomorphism \(\Gr(m,m+k) \cong \Gr(k, m+k)\), which comes from the reflection symmetry of the \(A_{m+k-1}\) Dynkin diagram:
        \begin{center}
                \begin{dynkinDiagram}[text style/.style={scale=0.9},scale=1.5]{A}{oo*ooo}\end{dynkinDiagram}\(\ \ \ \cong\ \ \ \)\begin{dynkinDiagram}[text style/.style={scale=0.9},scale=1.5, backwards]{A}{oo*ooo}\end{dynkinDiagram}.
        \end{center}
\end{example}

\begin{example}\label{example:D4}
        \begin{equation*}
                \cP_{Q^6}=\tableau{12}{
                [a]&[a]&[a]&[]\\
                []&[a]&[a]&[a]
                }\ \ \cong\ \
                \tableau{12}{[a]&[a]&[a]\\[]&[a]&[a]\\[]&[]&[a]}=\cP_{\OG(4,8)},
        \end{equation*}
        therefore, \(Q^6\cong \OG(4,8)\). This isomorphism comes from the rotation symmetry of the \(D_4\) Dynkin diagram:
        \begin{center}
                \begin{dynkinDiagram}[text style/.style={scale=0.9},scale=1.5]{D}{*ooo}\end{dynkinDiagram}\(\cong\ \ \)
                \begin{dynkinDiagram}[text style/.style={scale=0.9},scale=1.5]{D}{ooo*}\end{dynkinDiagram}.
        \end{center}
\end{example}

\begin{example}
        Using \Table{poset}, it is not hard to see that if a Grassmannian Schubert variety is isomorphic to a non-type \(A\) cominuscule Schubert variety, then they are both isomorphic to a projective space. Indeed, in order to fit inside a \(\cP_X\) of another type, the lower order ideal is forced to be a chain.

        As a special case, we also see that any Schubert curve in any cominuscule flag variety is isomorphic to \(\P^1\). In fact, any Schubert curve in any flag variety is isomorphic to \(\P^1\), which follows from the more general statements that Schubert varieties are rational normal projective varieties and that \(\P^1\) is the only rational normal projective curve.
\end{example}

\begin{example}
        The Schubert divisor in \(Q^3\) is not isomorphic to \(\P^2\), because the labeling of their posets does not match:
        \begin{equation*}
                \tableau{12}{
                        [a]&[a]s
                }\ \ \not\cong\ \
                \tableau{12}{
                        [a]&[a]
                }\ .
        \end{equation*}
        We can also see it geometrically, as the Schubert divisor in \(Q^3\) is singular.
\end{example}

\begin{example}
        The quadric \(Q^3\) embeds in \(\LG(n,2n)\) (\(n\geq 3\)) as a Schubert variety, as illustrated by
        \begin{equation*}
                \tableau{12}{
                [AA]&[AA] s &[a] s &[a] s\\
                []&[AA]&[a]s &[a] s\\
                []&[]&[a]&[a]s\\
                []&[]&[]&[a]
                }\ .
        \end{equation*}
\end{example}

\begin{example}\label{example:Q10}
        The quadric \(Q^{10}\) embeds in \(E_7/P_7\) as a Schubert variety, as illustrated by
        \begin{equation*}
                \tableau{12}{
                [AA]&[AA]&[AA]&[AA]&[AA]&[a]\\
                []&[]&[]&[AA]&[AA]&[a]\\
                []&[]&[]&[]&[AA]&[a]&[a]\\
                []&[]&[]&[]&[AA]&[a]&[a]&[a]&[a]\\
                []&[]&[]&[]&[AA]&[a]&[a]&[a]&[a]\\
                []&[]&[]&[]&[]&[]&[]&[a]&[a]\\[]&[]&[]&[]&[]&[]&[]&[]&[a]\\[]&[]&[]&[]&[]&[]&[]&[]&[a]\\[]&[]&[]&[]&[]&[]&[]&[]&[a]
                }\ .
        \end{equation*}
\end{example}

\begin{example}
        There are two non-isomorphic \(6\)-dimensional Schubert varieties in \(E_6/P_6\), given by the two order ideals illustrated below.
        \begin{equation*}
                \tableau{12}{
                [AA]&[AA]&[AA]&[AA]\\
                []&[]&[AA]&[a]&[a]&[a]\\
                []&[]&[AA]&[a]&[a]&[a]\\
                []&[]&[]&[]&[a]&[a]&[a]&[a]
                }\ \ \not\cong\ \
                \tableau{12}{
                [AA]&[AA]&[AA]&[AA]\\
                []&[]&[AA]&[AA]&[a]&[a]\\
                []&[]&[a]&[a]&[a]&[a]\\
                []&[]&[]&[]&[a]&[a]&[a]&[a]
                }
        \end{equation*}
\end{example}

\begin{example}
        While \(\cP_{\LG(n,2n)}\) and \(\cP_{\OG(n+1,2n+2)}\) are isomorphic as posets, this isomorphism does not preserve the labeling of long/short roots (see illustration below). As a result, \(\LG(n,2n)\) and \(\OG(n+1,2n+2)\) do not contain isomorphic Schubert varieties of dimension greater than one.
        \begin{equation*}
                \tableau{12}{
                [a]&[a] s &[a] s &[a] s\\
                []&[a]&[a]s &[a] s\\
                []&[]&[a]&[a]s\\
                []&[]&[]&[a]
                }\ \ \not\cong\ \
                \tableau{12}{
                [a]&[a] &[a] &[a]\\
                []&[a]&[a] &[a] \\
                []&[]&[a]&[a]\\
                []&[]&[]&[a]
                }
        \end{equation*}
\end{example}

\section{Preliminaries}\label{sec:prelim}

Let \(G\) be a complex reductive linear algebraic group. We fix subgroups \(T \subset B \subset G\), where \(T\) is a maximal torus and \(B\) is a Borel subgroup. With this setup, \(T \subset G\) determines a root system \(R\) of \(G\), with corresponding Weyl group \(W:= N(T)/T\), and \(B\) determines a set of simple roots \(\Delta\subseteq R\). The set of roots decomposes into positive and negative roots $R=R^+\sqcup R^-$, with \(R^+\) being non-negative sums of simple roots. The Weyl group \(W\) is generated by the set of simple reflections
\[
        S\coloneqq\{s_\alpha: \alpha\in\Delta\}.
\]

To each subset \(I \subseteq S\) one can associate a Weyl subgroup \(W_I := \langle s : s \in I\rangle \subseteq W \), a parabolic subgroup \(P_I = B W_I B \subseteq G\) and the corresponding (partial) flag variety \(X = G/P_I\). Schubert varieties in \(X\) are indexed by \(W^I\), the set of minimal length coset representatives of \(W/W_I\). Explicitly, for \(w\in W^I\), the Schubert variety
\[
        X_w:=\overline{BwP_I/P_I}
\]
has dimension the Coxeter length of \(w\), denoted \(\ell(w)\). Moreover, for any \(u\in W^I\), we have \(X_u\subseteq X_w\) if and only if \(u\leq w\) in Bruhat order.

From now on, \(X\) is a cominuscule flag variety. In other words, \(I = S \setminus \{s_\gamma\}\), where \(\gamma\) is a cominuscule simple root, i.e., \(\gamma\) appears with coefficient \(1\) in the highest root of \(R\). Cominuscule roots are illustrated by filled-in circles in \Table{comin} and \Table{poset}.

Recall that
\[
        \cP_X\coloneqq \{\alpha\in R: \alpha\geq\gamma\}
\]
inherits the usual partial order on roots, i.e., $\alpha\leq \beta$ if $\beta-\alpha$ is a non-negative sum of simple roots, and in addition, we give \(\cP_X\) a labeling of long/short roots.

In \cite{Proctor84}, Proctor proves that \(W^I\) is a distributive lattice under the induced Bruhat partial order from $W$. Birkhoff's representation theorem implies there is a bijection between $W^I$ and the set of lower order ideals in $\cP_X$. In particular, the join-irreducible elements of $W^I$ are identified with principal lower order ideals of $\cP_X$ and hence with $\cP_X$ itself. See \Figure{bruhat} for an illustration when \(X=\Gr(2,4)\). Explicitly, to each \(w\in W^I\) we associate its inversion set
\begin{equation}\label{eqn:inversion_set}
        \lambda(w)\coloneqq \{\alpha\in R^+: w.\alpha<0\},
\end{equation}
viewed as a sub-poset of \(\cP_X\). It is well known that $\ell(w)=|\lambda(w)|$. Moreover, the following proposition was proved in \cite[Proposition 2.1 and Lemma 2.2]{ThYo09} and \cite[Theorem 2.4 and Corollary 2.6]{buch2016k}:

\begin{prop}[Thomas--Yong, Buch--Samuel]\label{prop:comin_properties}
        For any \(w\in W^I\), the inversion set $\lambda(w)$ is a lower order ideal in $\cP_X$.  Moreover:
        \begin{enumerate}
                \item The map $w\mapsto \lambda(w)$ is a bijection between $W^I$ and the set of lower order ideals in $\cP_X$.
                \item For any $u\in W^I$, we have $u\leq w$ in Bruhat order if and only if $\lambda(u)\subseteq \lambda(w)$.
                \item If $\alpha\in \lambda(w)$ and $\lambda(w)\setminus\{\alpha\}$ is a lower order ideal, then $ws_\alpha\in W^I$ and $\lambda(ws_\alpha)=\lambda(w)\setminus\{\alpha\}$, where $s_\alpha\in W$ is the reflection corresponding to \(\alpha\).
        \end{enumerate}
\end{prop}

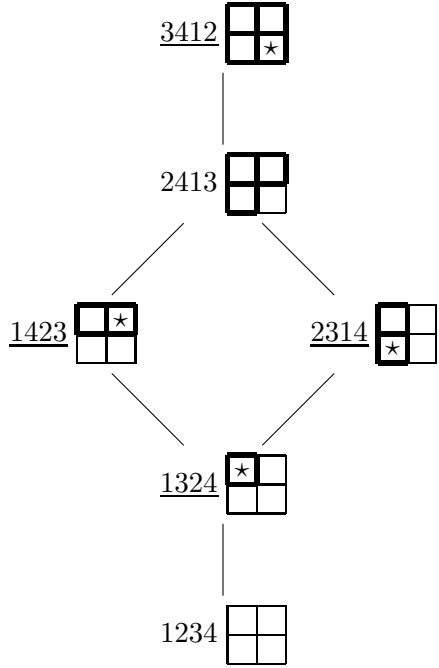
\begin{figure}[!ht]
        \caption{The Bruhat poset \(W^I\) when \(X=\Gr(2,4)\). Permutations in \(W^I\) are denoted using one-line notation, and next to each is the corresponding lower order ideal in \(\cP_X\). Join-irreducible elements of \(W^I\) are the ones underlined, and the generator of the corresponding principal lower order ideal in \(\cP_X\) is decorated with a \(\star\).}\label{fig:bruhat}
        \centering
        \begin{tikzpicture}
                \node (max) at (0,4) {$ \underline{3412}\  \tableau{11}{[AA] & [AA] \\ [AA] & [AA] \star}  $};
                \node (a) at (0,2) {$ 2413\  \tableau{11}{[AA] & [AA] \\ [AA] & [a]}$};
                \node (d) at (-2,0) {$\underline{1423}\  \tableau{11}{[AA] & [AA] \star \\ [a] & [a]}$};
                \node (e) at (2,0) {$\underline{2314}\  \tableau{11}{[AA] & [a] \\ [AA] \star & [a]}$};
                \node (f) at (0,-2) {$\underline{1324}\  \tableau{11}{[AA] \star & [a] \\ [a] & [a]}$};
                \node (min) at (0,-4) {$1234\  \tableau{11}{[a] & [a] \\ [a] & [a]}$};
                \draw (min) -- (f) -- (d) -- (a) -- (max);
                \draw (f)--(e)--(a);
        \end{tikzpicture}

\end{figure}

\begin{notation}
        Given a lower order ideal \(\lambda\subseteq\cP_X\), we will write \(w_\lambda\) for the element of \(W^I\) corresponding to \(\lambda\) in \Proposition{comin_properties}. We also use $X_\lambda:=X_{w_\lambda}$ to denote the corresponding Schubert variety.
\end{notation}

In Section \ref{sec:converse}, we will use a map \(\delta: \cP_X\to\Delta\) defined in \cite{buch2022positivity} as follows.

\begin{defn}
        For \(\alpha\in\cP_X\), let \(\lambda_\alpha\) be the principal lower order ideal generated by \(\alpha\). Let \(\delta(\alpha)=-w_{\lambda_\alpha}.\alpha\in R^+\). Then \(s_{\delta(\alpha)}=w_{\lambda_\alpha}s_\alpha w_{\lambda_\alpha}^{-1}\) has length \(1\) \cite{buch2022positivity}*{Section 4.1}. Therefore, \(\delta(\alpha)\in\Delta\).
\end{defn}

The following lemma is a restatement of \cite{buch2016k}*{Corollary 2.10}. See also \cite{buch2022positivity}*{Section 4.1}.

\begin{lemma}[Buch--Samuel]\label{lemma:red_decomp}
        Let \(\lambda\subseteq\cP_X\) be a lower order ideal and \(\beta_1,\beta_2,\dots,\beta_\ell\) be the boxes it contains written in an increasing order. Then \(s_{\delta(\beta_\ell)}\cdots s_{\delta(\beta_2)}s_{\delta(\beta_1)}\) is a reduced decomposition of \(w_\lambda\). Moreover, every reduced decomposition of \(w_\lambda\) can be obtained in this way.
\end{lemma}

\section{Proof of \Theorem{main}}\label{sec:proof}

In this section, we prove each direction of \Theorem{main} separately.

\subsection{Forward direction: the isomorphism class of \(X_\lambda\) determines the labeled poset \(\lambda\)}\label{sec:forward}

Let \(i_\lambda: X_\lambda \hookrightarrow X\) denote the embedding of a Schubert subvariety into a cominuscule flag variety \(X=G/P_I\).  The definition of the labeled poset $\lambda$ depends on the root system of the reductive group $G$ and hence on the embedding \(i_\lambda: X_\lambda \hookrightarrow X\).  The goal of this section is to show that $\lambda$ (as a labeled poset) can be constructed using only the variety structure of $X_\lambda$ and is therefore intrinsic to the isomorphism class of $X_\lambda$.  We prove the following proposition which states the ``forward" direction of \Theorem{main}.

\begin{prop}\label{prop:forward}
        Let $X_\lambda\subseteq X$ and $Y_{\mu}\subseteq Y$ be cominuscule Schubert varieties indexed by lower order ideals \(\lambda\subseteq\cP_X\) and \(\mu\subseteq\cP_Y\), respectively. If $X_\lambda$ and $Y_{\mu}$ are algebraically isomorphic, then $\lambda$ and $\mu$ are isomorphic as labeled posets.
\end{prop}

Our primary tools come from the intersection theory of algebraic varieties (see \cite{MR1644323} for more details).  Let \(\Pic(X_\lambda)\) and \(A_*(X_\lambda)\) denote the Picard and Chow groups of $X_\lambda$.  It is well known that these groups are algebraic invariants of \(X_\lambda\). Recall that the \(k\)-th Chow group \(A_k(Z)\) of a scheme \(Z\) is the free abelian group on the \(k\)-dimensional subvarieties of \(Z\) modulo rational equivalence. When \(Z\) is a normal variety, the Picard group \(\Pic(Z)\) can be identified with the subgroup of \(A_{\dim(Z)-1}(Z)\) generated by classes of locally principal divisors (note that all Schubert varieties are normal).  Our aim is to construct the labeled poset \(\lambda\) from the intersection class map or intersection product \cite[Definition 2.3]{MR1644323}:
\[
        \Pic(X_\lambda)\times A_*(X_\lambda)\to A_*(X_\lambda).
\]
If $(\sigma,\tau)\in \Pic(X_\lambda)\times A_*(X_\lambda)$, we denote the image of the intersection product by $\sigma\cdot\tau$.  Next, we consider the effective cone of a scheme:
\begin{defn}
        Let \(Z\) be a scheme. The \emph{effective cone} in the Chow group \(A_*(Z)\) is the semigroup in \(A_*(Z)\) generated by the classes of closed subvarieties of \(Z\).
\end{defn}
Since the flag variety $X$ is cominuscule, there is a unique Schubert variety of codimension \(1\) in \(X\), called the Schubert divisor.  Its class, denoted \(D\), is the unique effective generator of the Picard group \(\Pic(X)\subseteq A_*(X)\). Recall that \(i_\lambda: X_\lambda \hookrightarrow X\) is a closed embedding of varieties and let \(i_\lambda^*: \Pic(X)\to \Pic(X_\lambda)\) denote the induced map on Picard groups.  \Lemma{pic} below follows from \cite[Proposition 6]{Ol88} (see also \cite[Proposition 2.2.8 part (ii)]{Br05}).

\begin{lemma}\label{lemma:pic}
        For any non-empty lower order ideal \(\lambda\subseteq \cP_X\), the map \(i_\lambda^*: \Pic(X)\to \Pic(X_\lambda)\) is an isomorphism.
\end{lemma}

Since $D$ is effective and generates $\Pic(X)$, \Lemma{pic} implies that \(i_\lambda^*(D)\) is the unique effective generator of \(\Pic(X_\lambda)\).  Recall from \Proposition{comin_properties} that we have lower order ideals $\mu\subseteq \lambda$ if and only if $w_\mu\leq w_\lambda$ in Bruhat order.  Hence, we have $\mu\subseteq \lambda$ if and only if $X_\mu\subseteq X_\lambda$.  We write \([X_\mu]\) for the class of \(X_\mu\) in \(A_*(X_\lambda)\).  It is well known that the classes $\{[X_\mu]\}_{\mu\subseteq \lambda}$ form an integral basis of $A_*(X_\lambda)$.  \Lemma{schub_classes} below is a special case of \cite[Corollary of Thereom 1]{MR1299008} and allows us to identify Schubert classes (the effective cone) in \(A_*(X_\lambda)\).

\begin{lemma}[Fulton--MacPherson--Sottile--Sturmfels]\label{lemma:schub_classes}
        The Schubert classes \([X_\mu]\) such that \(\mu\subseteq \lambda\) are exactly the minimal elements in the extremal rays of the effective cone in \(A_*(X_\lambda)\).
\end{lemma}

We shall see later that the poset structure of \(\lambda\) can be recovered from the intersection products \(i_\lambda^*(D)\cdot [X_\mu]\).  Let $(i_{\lambda})_*:A_*(X_\lambda)\rightarrow A_*(X)$ denote the proper push-forward on Chow groups.  By the projection formula (\cite[Proposition 2.5 (c)]{MR1644323}), we have
\[
        (i_{\lambda})_*(i_\lambda^*(D)\cdot [X_\mu])=D\cdot (i_\lambda)_*([X_\mu]).
\]
Since $(i_{\lambda})_*$ is injective, the product \(i_\lambda^*(D)\cdot [X_\mu]\) in \(A_*(X_\lambda)\) can be computed via the product \(D\cdot (i_\lambda)_*([X_\mu])\) in \(A_*(X)\). By \cite[Example 19.1.11]{MR1644323}, the Chow group \(A_*(X)\) can be identified with the homology group \(H_*(X)\), with \((i_\lambda)_*([X_\mu])\) corresponding to the homology class of the Schubert variety $X_\mu\subseteq X$.  By \cite[Proposition 19.1.2]{MR1644323} we have that the intersection product \(D\cdot (i_\lambda)_*([X_\mu])\) can be identified with a cap product.  Since \(X\) is a smooth complex variety, the Poincar\'e duality  further identifies the intersection product with the cup product of cohomology classes corresponding to \(D\) and \((i_\lambda)_*([X_\mu])\). This cup product is given by the Chevalley formula \cite[Lemma 8.1]{fulton2004quantum}. Using these identifications, we restate the Chevalley formula for cominuscule flag varieties (and hence Schubert varieties):

\begin{lemma}[Fulton--Woodward]\label{lemma:chevalley}
        Let \(X\) be a cominuscule flag variety with corresponding cominuscule simple root $\gamma$.  For any lower order ideals \(\mu\subseteq\lambda\subseteq\cP_X\), let \([X_\mu]\) denote the class of $X_\mu$ in \(A_*(X_\lambda)\).  Then
        \begin{equation*}
                \label{eqn:mult_equation}
                i_\lambda^*(D)\cdot [X_\mu]=\sum \frac{(\gamma,\gamma)}{(\alpha,\alpha)}\, [X_{\mu\setminus\{\alpha\}}]
        \end{equation*}
        where the sum is over all positive roots \(\alpha\) such that \(\mu\setminus\{\alpha\}\) is a lower order ideal in \(\cP_X\).  Here $(\cdot,\cdot)$ denotes the usual inner product.
\end{lemma}

Observe that \Lemma{chevalley} reinterprets the Chevalley formula as a degree lowering operator since intersection product with divisors is a map from $A_k(X_\lambda)$ to $A_{k-1}(X_\lambda)$. This is opposite to the standard presentation of the Chevalley formula as a degree raising operator in cohomology.

\begin{example}
        By \Lemma{chevalley}, the following calculations hold for \(X=\LG(3,6)\).  We refer to \Table{comin} for the poset $\cP_X$.
        \begin{equation*}
                D\cdot\bracks[\bigg]{X_{\tableau{8}{
                [a]&[a]s&[a]s\\
                []&[a]}}}=2\bracks[\bigg]{X_{\tableau{8}{
                [a]&[a]s\\
                []&[a]}}}+\bracks[\bigg]{X_{\tableau{8}{
                                                [a]&[a]s&
                                                [a]s}}},
        \end{equation*}
        \begin{equation*}
                D\cdot\bracks[\bigg]{X_{\tableau{8}{
                [a]&[a]s\\
                []&[a]}}}=\bracks[\bigg]{X_{\tableau{8}{
                                                [a]&[a]s}}},
        \end{equation*}
        \begin{equation*}
                D\cdot\bracks[\bigg]{X_{\tableau{8}{
                                                [a]&[a]s}}}=2\bracks[\bigg]{X_{\tableau{8}{[a]}}}.
        \end{equation*}
        \begin{remark}
                Note that a coefficient \(2\) occurs whenever the removed box (root) is short.  Otherwise the coefficient is \(1\).  This is due to the fact that cominuscule flag varieties only appear in Dynkin types that are at most ``doubly laced" (see \Table{poset}).
        \end{remark}
\end{example}

\begin{proof}[Proof of \Proposition{forward}]
        Let $\tilde X$ be a variety that is algebraically isomorphic to a cominuscule Schubert variety.  Let \(\tilde E:=\{E_1,\ldots, E_k\}\subseteq A_*(\tilde X)\) denote the set of minimal elements in the extremal rays of the effective cone.  By \Lemma{schub_classes}, these classes form the Schubert basis of $A_*(\tilde X)$. \Lemma{pic} implies there is a unique effective generator of $\Pic(\tilde X)$ which we denote by $Z$.  For any $E_i\in \tilde E$, consider the intersection product
        \[Z\cdot E_i=\sum_{j} c_{ij}\, E_j.\]
        Define a partial order on the set \(\tilde E\) via the covering relations
        \[
                E_j < E_i\text{ if and only if }c_{ij}\neq 0.
        \]
        If $\tilde X\simeq X_\lambda$ and \(i_\lambda: X_\lambda \hookrightarrow X=G/P_I\) is an embedding of a Schubert subvariety into a cominuscule flag variety, then $Z$ corresponds to $i_\lambda^*(D)$ under the identification \(\Pic(\tilde X) \simeq \Pic(X_\lambda)\).  \Lemma{chevalley} implies that the poset $\tilde E$ is isomorphic to the set of lower order ideals in \(\cP_X\) that is contained in \(\lambda\), ordered by inclusion. Hence, \(\tilde E\) can be identified with the Bruhat interval $\{u\in W^I: u\leq w_\lambda\}$ via \Proposition{comin_properties}.  Let $\tilde\lambda$ denote the sub-poset of join-irreducible elements in $\tilde E$. Our discussions in \Section{prelim} imply that $\tilde \lambda$ is poset isomorphic to $\lambda$. Hence, the poset is independent of the embedding \(i_\lambda: X_\lambda \hookrightarrow X\).

        We finish the proof by showing that the labeling of long/short roots can also be recovered from the Chevalley formula.  Let $E_i\in\tilde\lambda$ and hence $E_i$ is join-irreducible in the poset $\tilde E$.  First, if $E_i$ is the unique minimal element in \(\tilde \lambda\), then we label $E_i$ as long (this corresponds to the cominuscule simple root).  Otherwise, \Lemma{chevalley} implies that
        \[
                Z\cdot E_i=c_{ij} E_j
        \]
        for some unique $E_j\in \tilde E$ with $c_{ij}\neq 0$.  If $c_{ij}=1$, then we label $E_i$ as long.  If $c_{ij}\neq 1$, then we label $E_i$ as short.  If $\tilde X\simeq X_\lambda$, then \Lemma{chevalley} implies that this labeling of $\tilde\lambda$ corresponds to the labeling of long/short roots in $\lambda$.

        In conclusion, the labeled poset $\tilde\lambda$ only depends on the isomorphism class of $\tilde X$.  In particular, if two cominuscule Schubert varieties $X_\lambda$ and $Y_\mu$ are algebraically isomorphic, then $\lambda\simeq \mu$ as labeled posets.
\end{proof}

\subsection{Converse direction: the labeled poset \(\lambda\) determines the isomorphism class of \(X_\lambda\)}\label{sec:converse}

Let \(X=G/P_I\) be a cominuscule flag variety and $\lambda\subseteq\cP_X$ be a lower order ideal. In this section, we prove that the poset $\lambda$ and its labeling of long/short roots determine the isomorphism class of $X_\lambda$. More precisely, we prove the following proposition, which states the ``converse" direction of \Theorem{main}.

\begin{prop}\label{prop:converse}
        Let $X_\lambda\subseteq X$ and $Y_{\mu}\subseteq Y$ be cominuscule Schubert varieties indexed by lower order ideals \(\lambda\subseteq\cP_X\) and \(\mu\subseteq\cP_Y\), respectively. If $\lambda$ and $\mu$ are isomorphic as labeled posets, then $X_\lambda$ and $Y_{\mu}$ are algebraically isomorphic.
\end{prop}

Our strategy is to embed \(X_\lambda\) in a ``minimal" flag variety \(X'\) determined by the labeled poset \(\lambda\).

Recall that \(S\) is the set of simple reflections defined in \Section{prelim}.

\begin{defn}
        The \emph{support} of \(\lambda\) is defined as
        \[
                S(\lambda)\coloneqq\{s\in S: s\leq w_\lambda\}.
        \]
        Equivalently, \(S(\lambda)\) is the set of simple reflections appearing in any reduced decomposition of \(w_\lambda\).
\end{defn}

Every reduced decomposition of \(w_\lambda\), and in particular, \(S(\lambda)\), can be read out from the poset \(\lambda\) \cite[Section 4]{buch2022positivity}. The variety \(X'\) is constructed using \(S(\lambda)\) as follows. Let \(G'\) be the reductive subgroup of \(P_{S(\lambda)}\) with Weyl group \(W'\coloneqq W_{S(\lambda)}\) and \(P'\coloneqq G'\cap P_I\) be the reductive subgroup of \(G'\) corresponding to \(I'\coloneqq I\cap S(\lambda)\). Set \(X'\coloneqq G'/P'\). Note that \(w_\lambda\in {W'}^{I'}\).

\Lemma{isom} below is a restatement of \cite[Lemma 4.8]{MR3552231}.

\begin{lemma}[Richmond--Slofstra]\label{lemma:isom}
        The inclusion \(X'\hookrightarrow X\) induces an isomorphism \(X'_{w_\lambda}\to X_\lambda\).
\end{lemma}

Let \(Y\) be another cominuscule flag variety and \(\mu\subseteq\cP_Y\) be a lower order ideal. Next, we show that a labeled poset isomorphism between \(\lambda\) and \(\mu\) induces an isomorphism between \(X'\) and \(Y'\), which restricts to an isomorphism between \(X_{\lambda}\) and \(Y_{\mu}\). We shall see that \(X'\) and \(Y'\) are cominuscule and that this isomorphism is given by an isomorphism of their Dynkin diagrams.

In the following, let $\cD_X$ be the Dynkin diagram of \(X\) with vertex set \(\Delta_X\).

\begin{defn}
        The diagram \(\cD_X^\lambda\) is defined to be the full subgraph of \(\cD_X\) with vertex set
        \[
                \Delta_X^\lambda\coloneqq\{\alpha\in\Delta_X: s_\alpha\in S(\lambda)\}.
        \]
\end{defn}

\begin{defn}\label{defn:Pdelta}
        Let a \emph{Dynkin chain} in \(\cP_X\) be a chain \(\pi\subseteq\cP_X\) such that:
        \begin{enumerate}
                \item the set \(\pi\) is a lower order ideal;
                \item the lengths of roots in \(\pi\) are weakly decreasing.
        \end{enumerate}

        The lower order ideal \(\cP_X^\Delta\subseteq\cP_X\) is defined to be the union of all Dynkin chains in \(\cP_X\).
\end{defn}

In the proof of \Lemma{biject}, we shall see that Dynkin chains in \(\cP_X\) correspond to paths in \(\cD_X\) starting from the cominuscule root~\(\gamma\). Examples of \(\cP_X^\Delta\) are illustrated in \Figure{DChains}.

\begin{lemma}\label{lemma:biject}
        The restriction \(\delta:\cP_X^\Delta\to\Delta_X\) is a bijection.
\end{lemma}
\begin{proof}
        Note that the cominuscule root \(\gamma\) is the unique minimal root in \(\cP_X\) and \(\delta(\gamma)=\gamma\). Also, since \(\gamma\) is a long root, it does not obstruct condition (2) in \Definition{Pdelta} and is an element of any non-empty Dynkin chain. Moreover, if \(\gamma, \beta_1,\dots,\beta_{m-1},\beta_m\) are distinct simple roots along a path in \(\cD_X\), then
        \[
                \gamma<(\gamma+\beta_1)<\dots<(\gamma+\beta_1+\dots+\beta_m)
        \]
        is a Dynkin chain in \(\cP_X\) and \(\delta(\gamma+\beta_1+\dots+\beta_m)=\beta_m\). This proves surjectivity.

        To prove injectivity, note that for \(\beta\), \(\beta'\) in \(\cP_X\), if \(\delta(\beta)=\delta(\beta')\), then \(\beta\) and \(\beta'\) are comparable \cite[Remark 4.2(b)]{buch2022positivity}. Since Dynkin chains are lower order ideals (\Definition{Pdelta} part (1)), it suffices to prove that each Dynkin chain in \(\cP_X^\Delta\) maps injectively to \(\Delta_X\). Let \(
        \gamma_0<\gamma_1 < \gamma_2 < \dots < \gamma_k
        \) be a Dynkin chain in \(\cP_X^\Delta\). (While we will not need it, we must have \(\gamma_0=\gamma\) because a Dynkin chain is a lower order ideal and \(\gamma\) is the unique minimal root in \(\cP_X\).) By \cite[Remark 4.2(c)]{buch2022positivity} we have that \(\delta(\gamma_0), \delta(\gamma_1), \dots, \delta(\gamma_k)\) forms a walk on \(\cD_X\). For a contradiction, assume this walk is not a path. Since \(\cD_X\) does not contain any cycles, we can assume without loss of generality that \(\delta(\gamma_{i-1})=\delta(\gamma_{i+1})\) for some \(i\), where \(0< i<k\). By definition,
        \[
                \delta(\gamma_j)=-w_j.\gamma_j\text{ for all }j,
        \]
        where \(w_j\coloneqq w_{\lambda_{\gamma_j}}\). By \Lemma{red_decomp}, we have
        \[
                w_i=s_{\delta(\gamma_i)}w_{i-1}\text{\quad and\quad  }w_{i+1}=s_{\delta(\gamma_{i+1})}s_{\delta(\gamma_i)}w_{i-1}.
        \]
        Therefore,
        \begin{equation}\label{eqn:gammas}
                \begin{aligned}
                        \gamma_{i-1} & =-w_{i-1}^{-1}.\delta(\gamma_{i-1})=-w_{i-1}^{-1}.\delta(\gamma_{i+1}),                                                                \\
                        \gamma_i     & =-w_{i-1}^{-1}s_{\delta(\gamma_i)}.\delta(\gamma_i)=w_{i-1}^{-1}.\delta(\gamma_i),                                                     \\
                        \gamma_{i+1} & =-w_{i-1}^{-1}s_{\delta(\gamma_i)}s_{\delta(\gamma_{i+1})}.\delta(\gamma_{i+1})=w_{i-1}^{-1}s_{\delta(\gamma_i)}.\delta(\gamma_{i+1}).
                \end{aligned}
        \end{equation}
        Note that
        \(
        (\delta(\gamma_j),\delta(\gamma_j))=(\gamma_j,\gamma_j)
        \) for all \(j\). Hence, using \eqn{gammas} we have
        \[
                (\delta(\gamma_{i-1}),\delta(\gamma_{i-1}))=(\delta(\gamma_{i+1}),\delta(\gamma_{i+1})),\]and by the decreasing condition on lengths (\Definition{Pdelta}), we must have \[(\delta(\gamma_{i-1}),\delta(\gamma_{i-1}))=(\delta(\gamma_i),\delta(\gamma_i))=(\delta(\gamma_{i+1}),\delta(\gamma_{i+1})).\] Since the simple roots \(\delta(\gamma_{i+1})\) and \(\delta(\gamma_i)\) are adjacent in the Dynkin diagram and of the same length,
        \[
                \langle\delta(\gamma_{i+1}),\delta(\gamma_i)\rangle\coloneqq\frac{2(\delta(\gamma_{i+1}),\delta(\gamma_i))}{(\delta(\gamma_i),\delta(\gamma_i))}=-1.
        \]
        As a consequence,
        \begin{equation}\label{eqn:deltas}
                s_{\delta(\gamma_i)}.\delta(\gamma_{i+1})=\delta(\gamma_{i+1})-\langle\delta(\gamma_{i+1}),\delta(\gamma_i)\rangle\delta(\gamma_i)=\delta(\gamma_{i+1})+\delta(\gamma_i).
        \end{equation}
        By \eqn{gammas} and \eqn{deltas}, we have
        \[
                \gamma_{i+1}-\gamma_{i}=w_{i-1}^{-1}.\delta(\gamma_{i+1})=-\gamma_{i-1},
        \]
        which is a contradiction.
\end{proof}

\begin{figure}[!ht]
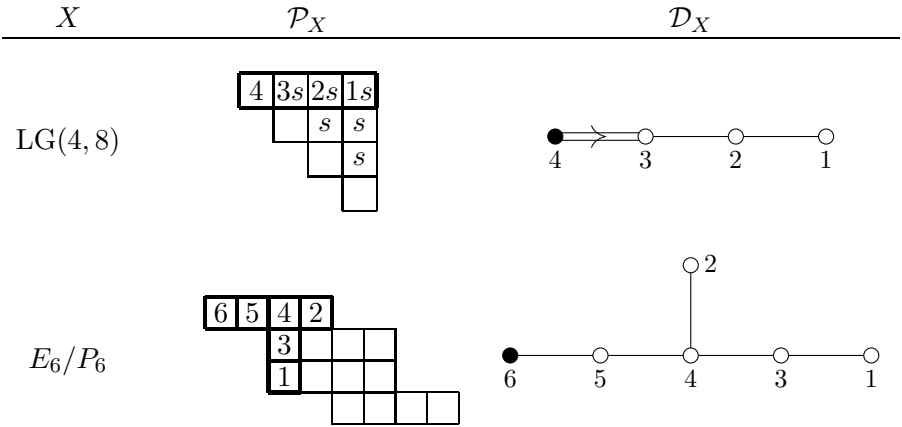

        \caption{We highlight \(\cP_X^\Delta\subset \cP_X\) with a bold border and label its boxes by the images of \(\delta\).
        }\label{fig:DChains}
        \centering

        \begin{tabular}{ccc}
                $X$        & $\cP_X$                                                                                                  & $\cD_X$                                    \\ \hline & & \\
                $\LG(4,8)$ & $\tableau{13}{
                [aHL] 4    & [aH] 3s                                                                                                  & [aH] 2s & [aHR] 1s                         \\
                []         & [a]                                                                                                      & [a]s    & [a] s                            \\
                []         & []                                                                                                       & [a]     & [a]s                             \\
                []         & []                                                                                                       & []      & [a]
                } $        & $\begin{dynkinDiagram}[text style/.style={scale=0.9},labels={1,2,3,4},scale=2, backwards]
                                              {C}{ooo*}
                                      \end{dynkinDiagram}$                                               \\ & & \\
                $E_6/P_6$  & \qquad  $\tableau{12}{
                [Aa]6      & [Aa]5                                                                                                    & [Aa]4   & [Aa]2                            \\
                []         & []                                                                                                       & [Aa]3   & [a]      & [a] & [a]             \\
                []         & []                                                                                                       & [Aa]1   & [a]      & [a] & [a]             \\
                []         & []                                                                                                       & []      & []       & [a] & [a] & [a] & [a]
                }$  \qquad & $\begin{dynkinDiagram}[text style/.style={scale=0.9},labels={6,2,5,4,3,1},scale=2]
                                              {E}{*ooooo}
                                      \end{dynkinDiagram}$
        \end{tabular}
\end{figure}

\begin{cor}\label{cor:restrict}
        Let \(\lambda\subseteq\cP_X\) be a lower order ideal. Then \(\delta(\lambda\cap\cP_X^\Delta)=\Delta_X^\lambda\), and \(\cD_X^\lambda\) is a connected Dynkin diagram.
\end{cor}
\begin{proof}
        By \cite{buch2022positivity}*{Remark 4.2.(b)}, for each \(\alpha\in\Delta_X\), there is a unique minimal root \(\beta\in \cP_X\) such that \(\delta(\beta)=\alpha\). From \Lemma{biject} we see that \(\beta\in\cP_X^\Delta\). It follows that \(\delta(\lambda)=\delta(\lambda\cap\cP_X^\Delta)\). By \Lemma{red_decomp}, we have \(\delta(\lambda)=\Delta_X^\lambda\). Therefore, \(\delta(\lambda\cap\cP_X^\Delta)=\Delta_X^\lambda\). Note that \(\lambda\cap\cP_X^\Delta\) is exactly the union of Dynkin chains contained in \(\lambda\), and by the proof of \Lemma{biject}, the map \(\delta\) sends each Dynkin chain to a path in \(\cD_X\) starting from \(\gamma\). Therefore, \(\cD_X^\lambda\) is connected.
\end{proof}

\begin{remark}
        \Corollary{restrict} implies that \(X'\) is the cominuscule flag variety given by the pair \((\cD_X^\lambda, \gamma)\).
\end{remark}

The last ingredient is \Proposition{group_restriction}, a purely combinatorial result. Geometrically, it implies that the ``minimal'' cominuscule flag varieties for Schubert varieties with isomorphic labeled posets are isomorphic.

\begin{prop}\label{prop:group_restriction}
        Let \(\lambda\subseteq\cP_X\) and \(\mu\subseteq\cP_Y\) be lower order ideals. Then every labeled poset isomorphism between \(\lambda\) and \(\mu\) induces a graph isomorphism between \(\cD_X^\lambda\) and \(\cD_Y^\mu\) that identifies reduced decompositions of \(w_\lambda\) and \(w_\mu\).
\end{prop}

\begin{proof}
        Let \(\lambda\subseteq\cP_X\) and \(\mu\subseteq\cP_Y\) be lower order ideals and \(f: \lambda\to\mu\) a labeled poset isomorphism. Note that \(f\) restricts to a labeled poset isomorphism \(f: \lambda\cap\cP_X^\Delta\to\mu\cap\cP_Y^\Delta\), since \(\lambda\cap\cP_X^\Delta\) is exactly the union of Dynkin chains contained in \(\lambda\). Define a map \(\psi_f: \Delta_X^\lambda\to\Delta_Y^\mu\) by \(\psi_f(\alpha)=\delta\circ f\circ\delta^{-1}(\alpha)\), where \(\delta^{-1}: \Delta_X\to\cP_X^\Delta\) is the inverse of \(\delta: \cP_X^\Delta\to\Delta_X\). Then \(\psi_f\) is a bijection.

        By the proof of \Lemma{biject}, simple roots \(\alpha\) and \(\beta\) are adjacent vertices in \(\cD_X\) if and only if either \(\delta^{-1}(\alpha)<\delta^{-1}(\beta)\) or \(\delta^{-1}(\beta)<\delta^{-1}(\alpha)\) is a covering relation. Since \(\delta\) preserves root lengths, the map \(\psi_f\) can be extended to a graph isomorphism \(\cD_X^\lambda\to\cD_Y^\mu\).

        The identification of reduced decompositions now follows from \Lemma{red_decomp}.
\end{proof}

\begin{proof}[Proof of \Proposition{converse}]
        Let $\gamma_X$ and $\gamma_Y$ denote the cominuscule simple roots corresponding to the cominuscule flag varieties $X$ and $Y$.  Let $X'$ and $Y'$ denote the cominuscule flag varieties given by the pairs \( (\cD_X^\lambda, \gamma_X)\) and \((\cD_Y^\mu, \gamma_Y)\).  \Proposition{group_restriction} implies the cominuscule flag varieties \(X'\) and \(Y'\) are isomorphic.  By \Lemma{isom}, this isomorphism restricts to an isomorphism between \(X_\lambda\) and \(Y_\mu\), upon identifying them with Schubert varieties in \(X'\) and \(Y'\), respectively.
\end{proof}

We illustrate the above process with
\Example{proof} and \Example{E6P5} below.

\begin{example}\label{example:proof}
        Let \(X = E_6/P_6\) and \(\lambda\) be the lower order ideal depicted on the left below. Then \(S(\lambda) = \{s_2, s_3, s_4, s_5, s_6\}\), where \(s_i\) is the simple reflection corresponding to the simple root labeled by \(i\). Therefore, the pair \((\cD_X^\lambda,\gamma)\) is as depicted on the right, isomorphic to that of \(Q^8\), showing \(X' \cong Q^8\), and \(X_\lambda\cong X'_{\lambda'}\), where \(\lambda'\) is the lower order ideal depicted on the right below.
        
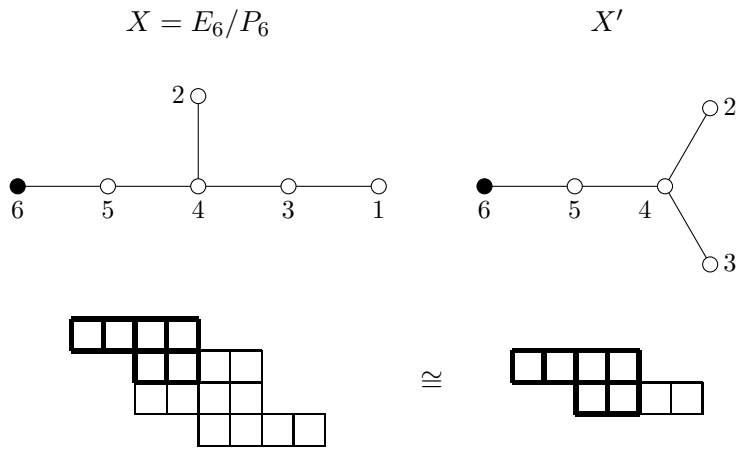
\begin{table}[h]
                \centering
                \begin{tabular}{c c c}
                        \(X = E_6/P_6\)                                                                                         &         & \(X'\)                                \\ & & \\
                        \begin{dynkinDiagram}[text style/.style={scale=0.9},labels={1,2,3,4,5,6},scale=2,backwards]
                                {E}{ooooo*}
                        \end{dynkinDiagram} &         &
                        \begin{dynkinDiagram}[text style/.style={scale=0.9},labels={6,5,4,2,3},scale=2]
                                {D}{*oooo}
                        \end{dynkinDiagram}
                        \\ & & \\
                        \(
                        \tableau{12}{
                        [AA]                                                                                                    & [AA]    & [AA]   & [AA]                         \\
                                                                                                                                &         & [AA]   & [AA] & [a] & [a]             \\
                                                                                                                                &         & [a]    & [a]  & [a] & [a]             \\
                                                                                                                                &         &        &      & [a] & [a] & [a] & [a]
                        }
                        \)                                                                                                      & $\cong$ &
                        \(\tableau{12}{
                        [AA]                                                                                                    & [AA]    & [AA]   & [AA]                         \\
                                                                                                                                &         & [AA]   & [AA] & [a] & [a]
                        }\)                                                                                                                                                       \\
                \end{tabular}
        \end{table}
\end{example}

\begin{example}\label{example:E6P5}
        Let \(X = E_6/P_6\) and \(\lambda\) be the lower order ideal depicted on the left below. Then \(S(\lambda) = \{s_1, s_3, s_4, s_5, s_6\}\), where \(s_i\) is the simple reflection corresponding to the simple root labeled by \(i\). Therefore, the pair \((\cD_X^\lambda,\gamma)\) is as depicted on the right, isomorphic to that of \(\P^5\), showing \(X_{\lambda} \cong X' \cong \P^5\).

        \begin{table}[h]
                \centering
                \begin{tabular}{c c c}
                        \(X = E_6/P_6\)                                                                                          &         & \(X'\)                                 \\ & & \\
                        \begin{dynkinDiagram}[text style/.style={scale=0.9},labels={1,2,3,4,5,6},scale=2,backwards]
                                {E}{ooooo*}
                        \end{dynkinDiagram} &         &
                        \begin{dynkinDiagram}[text style/.style={scale=0.9},labels={1,3,4,5,6},scale=2,backwards]
                                {A}{oooo*}
                        \end{dynkinDiagram}
                        \\ & & \\
                        \(
                        \tableau{12}{
                        [AA]                                                                                                     & [AA]    & [AA]   & [a]                           \\
                                                                                                                                 &         & [AA]   & [a]  & [a]  & [a]             \\
                                                                                                                                 &         & [AA]   & [a]  & [a]  & [a]             \\
                                                                                                                                 &         &        &      & [a]  & [a] & [a] & [a]
                        }
                        \)                                                                                                       & $\cong$ &
                        \(\tableau{12}{
                        [AA]                                                                                                     & [AA]    & [AA]   & [AA] & [AA]
                        }\)                                                                                                                                                         \\
                \end{tabular}
        \end{table}
\end{example}

\bibliography{comin_FPSAC.bib}
\end{document}